\documentclass[11pt]{amsart}
\usepackage{amssymb}
\usepackage{amsmath,amscd}

\newcommand{\A}{\mathbb{A}_K} 
\newcommand{\bbar}[1]{{\overline{#1}}} 
\newcommand{\C}{\mathbf C} 
\newcommand{\F}{\mathbf F} 
\newcommand{\floor}[1]{{\lfloor{#1}\rfloor}} 
\newcommand{\G}{\mathbf{G}} 
\newcommand{\gotik}[1]{\mathfrak{#1}} 
\newcommand{\legendre}[2]{{\left( \frac{#1}{#2}\right)}} 
\newcommand{\Q}{\mathbf Q} 
\newcommand{\SO}{{\mathbf{SO} }} 
\newcommand{\Tr}{{\mathrm{Tr}\, }} 
\newcommand{\Z}{\mathbf Z}

\renewcommand{\b}[1]{{\mathbf{#1}}} 
\renewcommand{\deg }{{\mathrm{deg\,} }} 
\newcommand{\GL}{{\mathbf{GL} }} 
\newcommand{\J}{{\mathbb J}}

\newcommand{\Pic}{{\mathrm{Pic}\, }} 
\newcommand{\pp}{\mathfrak{p}}

\newtheorem{theorem}{Theorem}[section] 
\newtheorem{lemma}[theorem]{Lemma} 
\newtheorem{corollary}[theorem]{Corollary} 
\newtheorem{proposition}[theorem]{Proposition}

\theoremstyle{definition} 
\newtheorem{definition}{Definition} 
\newtheorem{remark}{Remark}

\numberwithin{equation}{section}

\begin{document}

\title[Siegel's mass formula and averages of $L$-functions]{Siegel's mass formula and averages of Dirichlet 
$L$-functions over function fields}
\author{Piotr Maciak and Jorge Morales}
\address{{\'E}cole Polytechnique F{\'e}d{\'e}rale de Lausanne\\
EPFL-FSB-MATHGEOM-CSAG\\
Station 8, 1015 Lausanne\\
Switzerland}
\email{piotr.maciak@epfl.ch}
\address{Mathematics Department\\
Louisiana State University\\
Baton Rouge, Louisiana}
\email{morales@math.lsu.edu}
\begin{abstract} Let
$\mathfrak{G}$ be a genus of definite ternary lattices over $\F_q[t]$ of square-free determinant. In this paper we give self-contained and relatively elementary proofs 
of Siegel's formulas for the weighted sum of primitive representations numbers over the classes of
$\mathfrak{G}$
and for the mass of $\mathfrak{G}$. Our proof of the mass formula shows an interesting and seemingly new relation 
with certain averages of Dirichlet $L$-functions.
\end{abstract}


\maketitle

\section{Introduction} 

Let $\F_q$ be the finite field with $q$ elements, where $q$ is odd. Let
$K=\F_q(t)$ be the field of rational functions over $\F_q$, and $A=\F_q[t]$
the polynomial ring. We denote by $K_\infty$ the completion of $K$ ``at infinity'',
that is $K_\infty=\F_q((t^{-1}))$, the field of Laurent series in $t^{-1}$ over $\F_q$.

A quadratic space $(V,Q)$  over $K$ is {\em definite} if the extended space
$(V\otimes K_\infty, Q)$ is anisotropic, that is if $Q$ it has no nontrivial zeros over $K_\infty$.
If $(V,Q)$ is definite, then the dimension of $V$ is at most four since
the field $K_\infty$ is $C_2$ by a theorem of Lang \cite[Theorem 8]{Lang:1952fk}.

Let $D\in A$ be a square-free polynomial and let $(V,Q)$ be a ternary definite quadratic
space over $K$ that contains
integral lattices of determinant $D$. Let $\mathfrak{G}$ denote the genus of such lattices.
In this paper we give self-contained and relatively 
elementary proofs for the Siegel-Minkowski formulas 
for the weighted sum of primitive representation numbers over the classes in $\mathfrak G$ (see \eqref{E:repr} below)
 and the formula for the mass of $\mathfrak G$ (see \eqref{E:mass1} below) . Both formulas
are obtained directly in a completely explicit form that does not involve 
Euler products of local densities.

Let $L_1,\ldots,L_h$
be representatives of all isometry classes in $\mathfrak G$. Let $n_i=|SO(L_i)|$ for $i=1,\ldots,h$.
  For $a\in A$ we denote by $R(L_i,a)$ the number of {\em primitive} solutions of the equation $Q(\b x)=a$ with
$\b x\in L_i$. We begin by establishing a formula for the weighted sum of the $R(L_i,a)$
\begin{equation}\label{E:repr}
\sum_{i=1}^h \frac{1}{n_i} R(L_i,a)= \frac{1}{2^r} |\Pic(A[\sqrt{- aD}])|,
\end{equation}
where $a$ is prime to the determinant $D$ and such that $-a D$ is not a square in $K_\infty$ and $r$ is the number of irreducible polynomials
dividing $D$ (Theorem \ref{T:siegel}). Here $\Pic(A[\sqrt{- aD}])$ is the Picard group of the order $A[\sqrt{- aD}]$. Formula \eqref{E:repr} is established from the action of a suitable id{\`e}le group 
on a set of lattices of $\mathfrak{G}$ that represent $a$ primitively. Formula \eqref{E:repr}
was already known to Gauss \cite[\S 292]{Gauss:1986vn} for the number of primitive solutions of the equation $x^2+y^2+z^2=a$ over $\Z$. Interpretations of Gauss' formula in
terms of the Hurwitz quaternions have been given by Venkov \cite{Venkov:1922ys} and Rehm \cite{Rehm:1982fk}. Shemanske \cite{Shemanske:1986} generalized
this to the norm form of a quaternion algebra of class number one over $\Q$. Our approach uses a correspondence between lattices and quaternion
orders that goes back to work of Brandt and Latimer (see the thorough exposition by Llorente \cite{Llorente:2000sf} on this topic). 

Recall that the quantity $\sum_{i=1}^h 1/n_i$ is called the {\em mass}
of $\mathfrak G$. We obtain an explicit expression for the mass using an argument that is specific to function fields. We compute the limit of ``averages'' of $R(L_i,a)$ as $a$ varies over a fixed degree
\begin{equation}
\lim_{m\to \infty} \frac{1}{q^{3m}} \sum_{\scriptsize \begin{matrix} \deg a=2m +\delta+1\\ (a,D)=1 \end{matrix} }
R(L_i, a),
\end{equation}
where $\delta=\deg D$. We show that this limit does not depend on the representative $L_i$ of the genus. Denoting
this limit by $\ell$ and using \eqref{E:repr} we get
\[
\sum_{i=1}^h \frac{1}{n_i} = \frac{1}{2^r \ell} \lim_{m\to \infty} \frac{1}{q^{3m}} \sum_{\scriptsize \begin{matrix} \deg a=2m+\delta+1 \\ (a,D)=1 \end{matrix} } |\Pic(A[\sqrt{- aD}])|.
\]
The evaluation of the limit on the right-hand side reveals an interesting connection with work by Hoffstein and Rosen \cite{Hoffstein:1992ao} on averages of $L$ functions over function fields (see Theorem \ref{T:L_avg} and Remark \ref{R:L_avg}). The computation of this limit in our situation requires the use of the ``Riemann Hypothesis'' for curves over finite fields, which was first proved by A. Weil \cite{Weil:1948uq}.  Elementary proofs have since then  been offered by Stepanov \cite{Stepanov:1972uq} (for hyperelliptic curves, the only case needed here), Schmidt \cite{Schmidt:1973vn}, Bombieri \cite{Bombieri:1974ys}. See \cite[Appendix]{Rosen:2002kv} for an exposition of Bombieri's proof  and interesting historical remarks.

In order to state the final formula for the mass we introduce the function 
\[
M_e(s)=\sum_{\scriptsize \begin{matrix} d|e\\ d\ \mathrm{monic} \end{matrix} }\mu(d) |d|^{-s},
\]
where $\mu$ is the M{\"o}bius function and $e\in A\setminus \{0\}$.

With this notation, the formula for the mass that we obtain is

\begin{equation}\label{E:mass1}  
\sum_{i=1}^h \frac{1}{n_i}= \frac{q^{\delta} M_D(1) M_{D_0}(2)}{2^r(q^2-1)(2M_{D_0}(1) -M_{D_0}(2))},
\end{equation}
where $D$ is the determinant and $\delta$ is its degree. The polynomial $D_0$ is the product of prime divisors of $D$ at which
$Q$ is isotropic, and $D_1$ is the product of prime divisors of $D$ at which
$Q$ is anisotropic, and $r$ is the total number of prime divisors of $D$ (Theorem \ref{T:mass}).

In the last section we apply \eqref{E:mass1} to obtain an exact formula for the class number $h$ in the case 
where $D$ is irreducible (Theorem \ref{T:exactclassno}).

\section{Preliminaries and notation}

The following notation will be used throughout this paper:
\medskip

\begin{tabular}{ l c l}
 $\F_q$ & : & the finite field with $q$ elements, where $q$ is odd \\
 $A$ & : & the polynomial ring $\mathbb{F}_q[t]$\\
 $K$ & : & the field of fractions of $A$\\
 $\mathbb{A}_K$ & : & the ad{\`e}le ring of $K$\\
 $(V,Q)$ & : & a regular quadratic space over $K$\\
   $C(V,Q)$ & : & the Clifford algebra  of $(V,Q)$ \\
 $C_0(V,Q)$ & : &  the even Clifford algebra  of $(V,Q)$
\end{tabular}

\medskip

Recall that a quadratic form $Q$ over $K$ is {\em definite} if
it is anisotropic (i.e. does not have nontrivial zeros) over $K_\infty=\F_q((t^{-1}))$. Notice
that definite forms exist only in rank $\le 4$.

In this paper, we shall limit ourselves to the case $\dim_KV = 3$ and
all quadratic forms are assumed to be definite. This implies in particular
that $C_0(V,Q)$ is a quaternion algebra over $K$ ramified at $\pp=\infty$.

If $\pp$ is a prime of $A$ then $A_\pp$, $K_\pp$ will denote
$\pp$-adic completions of $A$, $K$ respectively. Similar notation will be used to denote localizations of any
considered modules.

If $L$ is an $A$-lattice in $V$, the \emph{determinant} $\det(L)$ of $L$ is defined as
\[
 \det (B(\b e_i,\b e_j))  \in A / (\F^{\times}_q)^2,
\]
where $B(\b x,\b y)=\frac{1}{2}(Q(\b x+\b y)-Q(\b x)-Q(\b y))$ and  $(\b e_1, \dots, \b e_n)$ is a basis of $L$ over $A$.

A lattice $L\subset V$ is called {\em maximal} if it is maximal for the property
that $B(L,L)\subset A$.  

Recall that two lattices $L$ and $L'$ of $V$ are in the {\em same genus} if their completions $L_\pp$ 
and $L'_\pp$ are isometric for all primes $\pp$ of $K$.

\begin{lemma} All maximal lattices in $V$ are in the same genus.

\end{lemma}
\begin{proof} This follows from \cite[Theorem 91:2]{OMeara:2000fk}.
\end{proof}

\begin{lemma}\label{L:eqconditions} Let $(V,Q)$ be a ternary definite quadratic space over $K$
and let $L\subset V$ be an integral lattice of square-free determinant $D$. 
Let $a\in A$ be a polynomial 
relatively prime to $D$. The following conditions are equivalent

\begin{enumerate}
\item\label{list:1} The equation $Q(\b x)=a$ has a solution with $\b x\in V$.
\item\label{list:2} The equation $Q(\b x)=a$ has a solution with $\b x\in V_\infty$.
\item\label{list:3}The equation $Q(\b x)=a$ has a primitive solution  $\b x\in L_\pp$ for all primes $\pp$ of $K$ (including $\pp=\infty$).
\item\label{list:4} The equation $Q(\b x)=a$ has a primitive solution $\b x\in M$, where  $M$ is some lattice in the genus of $L$.
\item\label{list:5} $-a D$ is not a square in $K_\infty$.
\item\label{list:6} The extension $K(\sqrt{-a D})/K$ does not split at $\pp=\infty$ (i.e.
is ``imaginary'').
\end{enumerate}

\end{lemma}
\begin{proof} (\ref{list:1})$\implies$ (\ref{list:2}) is trivial.
\medskip

(\ref{list:2}) $\implies$ (\ref{list:3}). Let $\pp$ be a finite prime. If $\pp\nmid D$ then $(L_\pp,Q)$ has a binary hyperbolic
orthogonal factor $\langle 1, -1\rangle$ and hence $(L_\pp,Q)$ represents primitively any
element of $A_\pp$. If $\pp\mid D$ then $(L_\pp,Q)$ has a binary unimodular orthogonal factor
and hence represents all $\pp$-units, in particular $a$. 
\medskip

(\ref{list:3}) $\implies$ (\ref{list:1}) follows from the Hasse-Minkowski Theorem.
\medskip

(\ref{list:3}) $\iff$ (\ref{list:4}) follows from \cite[Ch. 9, Theorem 5.1]{Cassels:1978rf}.
\medskip

(\ref{list:5}) $\iff$ (\ref{list:2}). It is clear that $a$ is represented by $V_\infty$ if and only if $F=Q\perp \langle -a\rangle$ 
is isotropic over $K_\infty$. By \cite[Proposition 4.24] {Gerstein:2008fk} this condition holds if and only if
$- aD\not\in {K_\infty^\times}^2$ or $S_\infty (F)=1$, where $S_\infty$ denotes the Hasse 
invariant at $\infty$. Since $Q$ is anisotropic over $K_\infty$ we have $S_\infty(Q)=-(-1,D)_\infty$
(see e.g. \cite[Proposition 4.21] {Gerstein:2008fk}). Hence $S_\infty (F)=S_\infty(Q)(D,-a)_\infty =-(D,a)_\infty$. If $-a D\in {K_\infty^\times}^2$ then
$S_\infty (F)=-1$, which would imply that $a$ is not represented by $V_\infty$.
\medskip

(\ref{list:5}) $\iff$ (\ref{list:6}) is trivial.
\end{proof}

\section{Correspondence between lattices and orders}

We assume throughout this section that $(V,Q)$ is a ternary definite quadratic space over $K$.

Recall that the Clifford algebra $C(V,Q)$ is the quotient of the tensor algebra $T(V)$ by
the two-sided ideal generated by the set $\{\b x\otimes\b x-Q(\b x) \: : \: \b x\in V\}$. We refer 
to \cite[Chapter 9, \S2]{Scharlau:1985jv} or \cite[Chapter 10, \S2]{Cassels:1978rf} for the general properties of this construction.

If $L\subset V$ is an integral $A$-lattice we denote by $C(L)$ the Clifford algebra of $L$, which we view as
the $A$-order of $C(V,Q)$ generated by $L$. Its center
$Z(C(L))$ is a quadratic order over $A$, hence of the form
$A[\delta_L]$, where $\delta_L^2\in A$. Notice that $\delta_L$ 
is defined only up to a multiplicative constant. It is easy to see that 
$\delta_L^2 =-\det L$ (up to squares of $\F_q$).

\begin{lemma} Let $L,M\subset V$ be integral lattices with the same determinant.
Then $\delta_L=\delta_M u$, where $u\in \F_q^\times$.
\begin{proof}  $\delta_L$ and $\delta_M$ generate the quadratic extension 
$Z(C(V,Q))/K$ so $\delta_L=\delta_M u$ with $u\in K^\times$. Since $\delta_M^2=-\det M\equiv
-\det L=\delta_L^2 \pmod {\F_q^\times}^2$, we have $u\in \F_q^\times$.

\end{proof}

\end{lemma}

We denote by $C_0(V,Q)$ the even Clifford algebra of $(V,Q)$ (see \cite[Chapter 9, \S2] {Scharlau:1985jv} for the definition).
The algebra $C_0(V,Q)$ comes naturally equipped with the trace form $\langle x, y\rangle=\frac{1}{2}\Tr (x\bar y)$,
where $\Tr $ is the reduced trace and $\:\bar\ \:$ is the canonical involution on $C_0(V,Q)$. The associated
quadratic form is the norm form of $C_0(V,Q)$ and will be denoted by $N$. 
If $L\subset V$
is an integral lattice, we denote by $C_0(L)$ its even Clifford algebra, which is an $A$-order in $C_0(V,Q)$.
 An order in $C_0(V,Q)$ is in particular an $A$-lattice,
so it makes sense to speak about its determinant with respect to $\langle\ ,\ \rangle$.

 \begin{proposition}\label{P:det}
 Let $L$ be an integral $A$-lattice in $(V,Q)$. Then
 \[
  \det (C_0(L)) = \det(L)^2.
 \]
\end{proposition}

\begin{proof}
 Let $\pp$ be a prime of $A$ and let $\{{\b e}_1, {\b e}_2, {\b e}_3\}$ be an orthogonal basis of $L_\pp$. It is
 easy to check that $\{1, {\b e}_1{\b e}_2, {\b e}_1{\b e}_3,
 {\b e}_2{\b e}_3\}$ is an orthogonal basis of $C_0(L_\pp)$. Then
 
 \[
 \begin{aligned}
 \det (C_0(L_\pp))&=N({\b e}_1{\b e}_2)N({\b e}_1{\b e}_3)N({\b e}_2{\b e}_3)\\
 &= {\b e}_1^2{\b e}_2^2 {\b e}_1^2{\b e}_3^2 {\b e}_2^2{\b e}_3^2\\
 &= Q({\b e}_1)^2Q({\b e}_2)^2Q({\b e}_3)^2\\
 &=\det(L_\pp)^2.
 \end{aligned}
 \]
 The result follows.
\end{proof}

\begin{lemma}\label{L:lattice recovery} Let $L^\sharp \subset V$ be
the lattice dual to $L$. Then
\[
\delta_L L^{\sharp} = {\Lambda_0},
\]
where ${\Lambda_0}=\{\, x \in
C_0(L)\: : \:\Tr (x)=0 \,\}$.
\end{lemma}

\begin{proof} It is enough to prove the equality $\delta L^{\sharp}={\Lambda_0}$ locally
at all primes.

Let $\pp$ be a prime of $A$ and let ${\b e}_1, {\b e}_2, {\b e}_3$ be an orthogonal basis
of $L_\pp$. Then
$\{{\b e}_1{\b e}_2, {\b e}_1{\b e}_3, {\b e}_2{\b e}_3\}$ is a basis of $({\Lambda_0})_\pp$. Let  ${\b e}_i^\sharp =Q({\b e}_i)^{-1} {\b e}_i$ for $i=1,2,3$. It is easy to see that $\{{\b e}_1^\sharp , {\b e}_2^\sharp ,{\b e}_3^\sharp\}$ is a basis of $L^\sharp_\pp$ (in fact the dual basis). Since $\delta_L=
u {\b e}_1{\b e}_2{\b e}_3$, where $u\in A_\pp^\times$, we see by direct calculation
$\delta_L {\b e}_1^\sharp = u {\b e}_2 {\b e}_3,\  \delta_L {\b e}_2^\sharp = -u {\b e}_1 {\b e}_3,\ \delta_L {\b e}_3^\sharp = u {\b e}_1 {\b e}_2$. This shows the desired equality.
\end{proof}

\begin{proposition}\label{P:injectivity}
Let $L, M$ be integral $A$-lattices of $V$. 
If $C_0(L)=C_0(M)$, then $M = L$.
\end{proposition}

\begin{proof} Recall that $\delta_L$ denotes the element of $C(L)$ such that
$Z(C(L))=A[\delta_L]$ and $\delta_L^2=-\det L$. With this notation, if $C_0(L)=C_0(M)$ then
by Lemma~\ref{L:lattice recovery}, we have
$\delta_{L} L^{\sharp} = \delta_{M} M^{\sharp}$.
Since $\delta_{L} = \delta_{M} \alpha$ for some $\alpha \in K^{\times}$, we have
$L =\alpha M$. Now, by Proposition \ref{P:det}, we have $\det L=\det M$. This immediately implies
$\alpha\in A^\times$ and hence $L=M$.

 \end{proof}
 
It is a standard fact (see e.g. \cite[Theorem 3.1 and Corollary 2, Ch. 10]{Cassels:1978rf}) 
that $u V u^{-1}= V$ for each $u \in C_0(V,Q)^{\times}$ and that
the map \mbox{$c_u : V \to V$} 
defined by $c_u(x) = u x u^{-1}$ is an automorphism of $(V,Q)$. Moreover, the map
$c : C_0(V,Q)^{\times} \to O(V)$ given by $u\mapsto  c_u$ is a group
homomorphism inducing the exact sequence
\begin{equation}\label{E:exact1}
1 \longrightarrow K^{\times} \overset{i}{\longrightarrow} C_0(V,Q)^{\times}
\overset{c}{\longrightarrow} SO(V,Q) \longrightarrow 1.
\end{equation}

\begin{proposition}
Let $L$ and $M$ be integral $A$-lattices in $V$. Then $L$ and $M$ are isometric
if and only if the $A$-orders $C_0(L)$ and $C_0(M)$ are conjugate in $C_0(V,Q)$.
\end{proposition}

\begin{proof}
Assume that $L$, $M$ are isometric. Then $M = \sigma L$ for some 
$\sigma \in SO(V,Q)$. By \eqref{E:exact1} we have $\sigma = c_u$ for some $u \in
C_0(V,Q)^{\times}$. Thus $M = u L u^{-1}$, which implies that
$C_0(M) = u C_0(L) u^{-1}$. Conversely, assume that $C_0(M) = u C_0(L) u^{-1}$.
Then
 \[
  C_0(u L u^{-1}) = u \mspace{2mu}C_0(L) \mspace{1mu} u^{-1} = C_0(M).
 \]
 By Proposition~\ref{P:injectivity}, $u L u^{-1} = M$. 
\end{proof}

\begin{theorem}\label{T:surjectivity}
 Let $D$ be the determinant of a maximal lattice in $(V,Q)$ and let $\Lambda$ be an $A$-order of
 $C_0(V,Q)$ such that $\det(\Lambda) = b^4 D^2$ for some $b \in A$. Then there exists an integral
 $A$-lattice $L\subset V$ such that
 $ \Lambda = C_0(L)$. 
\end{theorem}

\begin{proof}
Let $M\subset V$ be a maximal lattice and let $\delta=b \delta_M$, so $\delta^2=- b^2 D$.  
Define $L = \delta {\Lambda_0}^{\sharp}$, where ${\Lambda_0}=\{x\in\Lambda\: : \:\Tr(x)=0\}$ and let ${\Lambda_0}^{\sharp}$
be the dual of ${\Lambda_0}$ with respect to the norm form $N$. By
Lemma \ref{L:lattice recovery}, it is enough to show that $\delta L^\sharp = {\Lambda_0}$.
Let 
$C_0(V,Q)_0=\{x\in C_0(V,Q)\: : \:\Tr(x)=0\}$.  For $x\in C_0(V,Q)_0$ we have
$Q(\delta x)=(\delta x)^2=\delta^2 x^2=-\delta^2 N(x)$, so
multiplication by $\delta$ in $C(V,Q)$ induces an isomorphism of
quadratic spaces $(C_0(V,Q)_0, -\delta^2 N)\stackrel{\simeq}{\to} (V,Q)$. It follows
immediately from this observation that $\delta^{-1} L^\sharp= \delta^{-2} {({\Lambda_0}^\sharp)}^\sharp=
\delta^{-2} {\Lambda_0}$. Thus $\delta L^\sharp={\Lambda_0}$ as desired.

\end{proof}

\begin{corollary}\label{C:maxorders} Let $D\in A$ be a square-free polynomial and assume that
$V$ admits integral lattices of discriminant $D$. Suppose further that $Q$ is anisotropic at all prime divisors of $D$. Then an integral lattice $L\subset V$
is maximal if and only if its even Clifford algebra $C_0(L)$ is a maximal order. 
\end{corollary}

\begin{proof} Let $L$ be a maximal lattice in $V$. Then $\det (L)=D$ and therefore by Proposition \ref{P:det} we have $\det C_0(L)=D^2$.
Notice that $D^2$ is the determinant of a maximal order since all its prime divisors
are ramified in $C_0(V,Q)$ ($Q$ is anisotropic at all these primes) and $D$ is square-free. This
is enough to ensure that $C_0(L)$ is maximal.

Conversely, if $\Lambda$ is a maximal order, then by Theorem \ref{T:surjectivity}, $\Lambda=C_0(L)$ for some 
lattice $L$ of determinant $D$.

\end{proof}

\begin{proposition}\label{C:prim}
An element $a\in A$ relatively prime to $\det L$ is represented primitively by $L$ if and only if $C_0(L)$ contains
a primitive element $\lambda$ such that $\lambda^2 = -a D$.
\end{proposition}

\begin{proof} Let $\Lambda = C_0(L)$. 
Suppose that $a\in A$ is such that $a = \b x^2$ for some  primitive $\b x \in L$. 
 Since $a$ is prime to $\det L$, $\b x$ is primitive in $L^\sharp$ as well. It follows that $\delta \b x$ is primitive 
 in $\delta L^\sharp={\Lambda_0}$. Clearly $(\delta \b x)^2 = \delta^2 {\b x}^2 = - a D$. The converse is obvious.
 
\end{proof}

\section{An id{\`e}le action on primitive representations}\label{S:action on x}

As in previous sections, $(V,Q)$ denotes a ternary definite quadratic space over $K$. We assume throughout that the maximal
lattices in $V$ have square-free determinant that we denote by $D$. Recall that $\mathbb{A}_K$ denotes the
ad{\`e}le ring of $K$.

Let $a\in A$ relatively prime to $D$ and let $\b v\in V$
 be a fixed vector with $Q({\b x})=a$. We consider the set
 \[
\mathcal{L}_{\b v}=\left\{L\: : \:L\subset V\ \mathrm{integral\ lattice};\ \det L=D; \ L\ni \b v\ \mathrm{primitively}\right\}.
\]

We denote by $\SO(V,Q)$ the group of orientation-preserving automorphisms of $(V,Q)$ that
we view as an algebraic group defined over $K$.  Let $\G$ be the stabilizer of 
$\b v$, which is a closed subgroup of $\SO(V,Q)$. It is clear that $\G(\A)$ acts on $\mathcal{L}_{\b v}$ by restriction of the natural action
of $\SO(V,Q)(\A)$ on the set of all $A$-lattices of $V$.

\begin{proposition}\label{P:trans}
The action of $\G(\A)$ on $\mathcal{L}_{\b v}$ is transitive.
\end{proposition}

\begin{proof} Let $\pp$ be a prime of $A$ and let $L, L'\in \mathcal{L}_{\b v}$. Assume first $\pp \nmid a=Q(\b v)$. Then
$A_\pp \b v$ is an orthogonal factor of both $L_\pp$ and $L'_\pp$. Write
$L_\pp=A_\pp \b v\perp M_\pp$ and $L'_\pp=A_\pp \b v\perp M'_\pp$. Since $L_\pp$ and $L'_\pp$ are isometric, so must be
$M_\pp$ and $M'_\pp$ and hence there exists $\sigma_\pp\in \G (K_\pp)$ such that $\sigma_\pp L_\pp=L'_\pp$.

Assume now $\pp \mid a=Q(\b v)$. Since $\b v$ is primitive in $L_\pp$, there exists ${\b v}_2\in L_\pp$ such that
$B(\b v, {\b v}_2)\equiv 1\pmod \pp$ and $Q({\b v}_2)\equiv 0 \pmod \pp$ . Since $L_\pp$ is unimodular (recall that 
$a$ and $D$ are relatively prime), we can assume further by Hensel's Lemma that $Q({\b v}_2)=0$. Replacing ${\b v}_2$ by a suitable
scalar multiple, we can also assume $B(\b v, {\b v}_2)=1$. Hence $\b v$ and $\b v_2$ span a unimodular binary lattice $N_\pp$ of determinant $-1$. Taking the orthogonal complement of $N_\pp$ in $L_\pp$ we get a vector $\b v_3\in L_\pp$ such
that $\{\b v, \b v_2, \b v_3\}$ is a basis of $L_\pp$. The Gram matrix of $Q$ in this basis is
\begin{equation}\label{E:Gram}
\begin{pmatrix} a & 1 \\1 & 0\end{pmatrix}\perp\langle  -D \rangle.
\end{equation}
The same construction yields vectors $\b v'_2, \b v'_3\in L'_\pp$ such that $\{\b v, \b v'_2, \b v'_3\}$ is
a basis of $L'_\pp$ and the Gram matrix of $Q$ in this basis is also as in \eqref{E:Gram}. Clearly the linear map 
$\sigma_\pp :V_\pp\to V_\pp$
given by $\sigma_\pp \b v= \b v, \sigma_\pp \b v_2= \b v'_2, \sigma_\pp \b v_3= \b v'_3$ is in $\G (K_\pp)$ and
satisfies $\sigma_\pp L_\pp=L'_\pp$.
\end{proof}

Consider now the natural exact sequence of algebraic groups over $K$ (see e.g. \cite[Ch. 8, \S 23]{Knus:1998vn} for
the general case)
\begin{equation}\label{E:exact0}
\begin{CD}
1@>>>\GL_1(K)@>i>> \GL_1 (C_0(V,Q))@>c>> \SO(V,Q) @>>> 1,
\end{CD}
\end{equation}
where $i$ is the natural inclusion and $c_u(\b x)=u \b x u^{-1}$. 
\medskip

Let $E=Z_{C_0(V,Q)}(\b v)$. Restricting $c$ to $\GL_1(E)$ we get an exact sequence of algebraic
groups over $K$
\begin{equation}\label{E:exact2}
\begin{CD}
1@>>>\GL_1(K)@>i>> \GL_1 (E)@>c>> \G @>>> 1.
\end{CD}
\end{equation}

\begin{lemma} Let $\J_K$ and $\J_E$ denote the id{\`e}le groups of $K$ and $E$ respectively. Then
we have an exact sequence
\begin{equation}\label{E:exact3}
\begin{CD}
1@>>>\J_K@>i>> \J_E@>c>> \G(\A) @>>> 1.
\end{CD}
\end{equation}
\end{lemma}
\begin{proof} The result follows easily by taking points over the ad{\`e}le ring $\A$ in \eqref{E:exact2}
and taking the natural identifications $\J_K=\GL_1(K)(\A)$ and $\J_E=\GL_1(E)(\A)$. The surjectivity of 
$c:\J_E\to \G(\A)$ is easy to verify directly.
\end{proof}

We let now $\J_E$ act on $\mathcal{L_\b v}$ via $c:\J_E\to \G(\A)$. It is clear from Proposition
\ref{P:trans} that this action is transitive. We shall now investigate the stabilizer of a lattice $L\in 
\mathcal{L_\b v}$ for this action. Let $B=A[\delta_L \b v]\subset E$.

\begin{proposition}\label{P:p-stab}
Let $L\in \mathcal{L_\b v}$ and let $\pp$ be a prime of $A$. Let $H_\pp=\{u\in E_\pp^\times\: : \:uL_\pp u^{-1}=
L_\pp\}$.
\begin{enumerate} \item\label{list2:1} If $\pp\nmid D$ then $H_\pp= K^\times_\pp B^\times_\pp$.
\item\label{list2:2} If $\pp \mid D$ then $H_\pp=E^\times_\pp$.
\end{enumerate}
\end{proposition}
\begin{proof} 

(\ref{list2:1}) Suppose $\pp\nmid D$. Then $\det C_0(L_\pp)$ is a $\pp$-unit by Proposition \ref{P:det}
and hence $C_0(L_\pp)$ is a maximal order in $C_0(V_\pp)\simeq M_2(K_\pp)$. Thus 
$H_\pp=(K^\times_\pp C_0(L_\pp)^\times)\cap E^\times_\pp=K^\times_\pp (C_0(L_\pp)^\times\cap E^\times_\pp)$. 
By the proof of Proposition \ref{C:prim}, $\delta \b v$ is primitive in $C_0(L_\pp)$, which implies $C_0(L_\pp)^\times\cap E^\times_\pp=
B^\times_\pp$.

(\ref{list2:2}) Suppose $\pp\mid D$. We have an orthogonal decomposition $V=K\b v\perp W$ and since $\pp\nmid a$ ($a$ and $D$ are assumed to be relatively prime),
we have a corresponding integral orthogonal decomposition
$L_\pp= A_\pp\b v\perp N_\pp$.  Hence we can identify $E_\pp$ with $C_0(W)$ and $B_\pp$ with $C_0(N_\pp)$.
Multiplication by $C_0(N_\pp)=B_\pp$ in $C(V,Q)$ makes $N_\pp$ into a $B_\pp$-module (see \cite{Kneser:1982lb})
 which is
free of rank one ($B_\pp$ is a maximal order since $D$ is square-free). It is easy to see by direct
computation that
for $u\in E_\pp=C_0(W_\pp)$ and $\b x\in W_\pp$ we have ${\b x} u= \bbar{u} \b x$, where $u\mapsto \bbar{u}$ is the 
canonical involution of $E_\pp$. 

For $u\in E_\pp^\times$ we have the following chain of equivalences
\[
\begin{aligned}
u \in H_\pp &\iff u N_\pp u^{-1}= N_\pp\\
&  \iff  u\bbar{u}^{-1} N_\pp=N_\pp\\
&  \iff  u\bbar{u}^{-1}\in B_\pp^\times.\\
\end{aligned}
\]

Since $B_\pp$ is ramified over $A_\pp$, all $u\in E_\pp^\times$ fulfill the last condition.

\end{proof}

We can now describe the stability subgroups for the action of $\J_E$ on $\mathcal{L_\b v}$.

\begin{proposition}\label{P:stab}
Let $L\in \mathcal{L_\b v}$ and let $S=\{\pp\: : \: \pp\mid D\}\cup\{\infty\}$. Then
\[
(\J_E)_L=\J_K (\prod_{\pp\in S} E_\pp^\times \times\prod_{\pp \not\in S} B_\pp^\times)
\]
\end{proposition}
\begin{proof} 
This is an immediate consequence of Proposition \ref{P:p-stab}.

\end{proof}

\begin{corollary}\label{C:pic}
The Picard group $\Pic(B[1/D])$  acts
simply transitively on the set $\mathcal{L_\b v}/E^\times$. In particular
\[
|\mathcal{L_\b v}/E^\times| = |\Pic(B[1/D])|.
\]

\end{corollary}
\begin{proof} It follows from Propositions \ref{P:trans} and \ref{P:stab} that $\J_E/E^\times(\J_E)_L$
acts simply transitively on $\mathcal{L_\b v}/E^\times$. Since $A$ is a PID,
we have
\[
\J_K= K^\times (K_\infty\times \prod_{\pp\neq \infty} A^\times_\pp)\subset E^\times  (\prod_{\pp\in S} E_\pp^\times \times \prod_{\pp \not\in S} B_\pp^\times).
\]

Using the inclusion above and Proposition \ref{P:stab}, we get

\[
E^\times(\J_E)_L= E^\times (\prod_{\pp\in S} E_\pp^\times \times \prod_{\pp \not\in S} B_\pp^\times),
\]
which proves the Corollary.

\end{proof}

\section{Weighted sum of representation numbers}

Let $L_1,L_2,\ldots,L_h\in \mathcal{L_\b v}$ be representatives of all isometry classes in $\mathcal{L_\b v}$. 
Let
\[
R_i=\{\b x\in L_i\: : \:Q(\b x)=a\ \mathrm{and}\ \b x\ \mathrm{is\ primitive\ in}\ L_i\}.
\]
We define a map $\psi_i : R_i\to \mathcal{L_\b v}/E^\times$ as follows: for $\b x\in R_i$, choose $\sigma\in SO(V,Q)$
such that $\sigma(\b x)=\b v$ and set
\begin{equation}\label{E:psi}
\psi_i(\b x)= (\sigma L_i),
\end{equation}
where $(\sigma L_i)$ denotes the image of $\sigma L_i$ in $\mathcal{L_\b v}/E^\times$. It is readily checked
that $(\sigma L_i)$ does not depend on the choice of $\sigma$. Indeed, if $\sigma,\tau \in SO(V,Q)$ are such
that $\sigma(\b x)= \tau(\b x)=\b v$, then $\sigma \tau^{-1} = c_u$ for some $u\in E^\times$ and hence
$\sigma L_i= c_u \tau L_i$. This shows $(\sigma L_i)=(\tau L_i)$.

The group $SO(L_i)$ acts naturally on $R_i$ and we denote by $R_i/SO(L_i)$ the orbit space for this action.

\begin{proposition}\label{P:cover}
Let $\psi_i : R_i\to \mathcal{L_\b v}/E^\times$ ($i=1,\ldots,h$) be the maps defined in \eqref{E:psi}.
\begin{enumerate}
\item\label{list3:1} Each $\psi_i$ induces an injective map
${\bbar\psi}_i : R_i/SO(L_i)\hookrightarrow \mathcal{L_\b v}/E^\times$.
\item\label{list3:2} $\mathcal{L_\b v}/E^\times$ is the disjoint union of the images $\psi_i(R_i)$ ($i=1,\ldots,h$).
\end{enumerate}
\end{proposition}
\begin{proof}
(\ref{list3:1}) It is clear from the definition \eqref{E:psi} of $\psi_i$ that $\psi_i(\rho \b x)= \psi_i(\b x)$ for $\rho\in
SO(L_i)$, so $\psi_i$ induces a map ${\bbar\psi}_i : R_i/SO(L_i)\to \mathcal{L_\b v}/E^\times$. We shall see that this map is injective. Suppose $\psi_i(\b x)=\psi_i(\b y)$ and let $\sigma,\tau\in SO(V,Q)$ be such that $\sigma(\b x)=\tau(\b y)=\b v$. Then, by the definition of $\psi_i$, we have $\sigma L_i=c_u \tau L_i$ for some $u\in E^\times$. Replacing
 $\tau$ by $c_u \tau$ does not affect the condition $\tau(\b y)=\b v$, so we can assume that $\sigma L_i=\tau L_i$.
 Thus $\tau^{-1}\sigma \in SO(L_i)$ and $\tau^{-1}\sigma(\b x)=\b y$. 
 
 (\ref{list3:2}) It is enough to observe that each of the sets $\psi_i(R_i)$ consists exactly of the classes of $\mathcal{L_\b v}/E^\times$ represented by lattices isometric to $L_i$, so these sets are pairwise disjoint. Since $L_1,\ldots,L_h$ represent
 all isometry classes in $\mathcal{L_\b v}$, their union is $\mathcal{L_\b v}/E^\times$.
\end{proof}

\begin{corollary}\label{C:gauss1} Let $\rho_i=|R_i/SO(L_i)|$, where $i=1,\ldots,h$. Then
\begin{equation}\label{E:gauss1}
\sum_{i=1}^h \rho_i =|\Pic(B[1/D])|,
\end{equation}
where $B=A[\delta\b v]\simeq A[\sqrt{-a D}]$.

\end{corollary}
\begin{proof} It follows immediately from Corollary \ref{C:pic} and Proposition \ref{P:cover}.
\end{proof}

\begin{lemma}\label{L:freeaction}
\begin{enumerate} 
\item\label{list4:1} If $\deg (a)>0$ then the group $SO(L_i)$ acts freely on $R_i$.
\item\label{list4:2} If $\deg (a)=0$ then the stabilizers $SO(L_i)_{\b x}$ have order $2$ for all ${\b x}\in R_i$.
\end{enumerate}

\end{lemma}

\begin{proof} (\ref{list4:1}) Set $L=L_i$ for simplicity of the notation. We have a natural embedding $SO(L)_\b v\hookrightarrow SO(M)$. Notice that $M$ cannot be 
unimodular since in this case $A\b v$ would be an orthogonal factor of $L$, which would cause $a$ and $D$ 
to differ by a unit of $A$, contradicting that $a$ and $D$ are relatively prime. So $M$ is a primitive 
definite binary lattice over $A$ with $\deg(\det M)>0$. It is an easy exercise to show that in this case 
$ SO(M)=\{\pm 1_M\}$. Hence $|SO(L)_{\b v}|\le 2$.

(\ref{list4:2}) Suppose that $a$ is a unit. Then $L=A\b v\perp M$ and $SO(L)_{\b v}\simeq  SO(M)$. Conversely, if $SO(L)_{\b v}\simeq  SO(M)$,
let $\sigma$ be the nontrivial element of $SO(L)_{\b v}$ and let $e = (1 + \sigma)/2$ and $f = (1-\sigma)/2$. It is clear that $e$, $f$ are idempotents and $e + f = 1$.
 Since $2$ is invertible in $A$, we have $L = e L + f L$ and this decomposition is orthogonal since $\sigma$ is an isometry. Also $A\b v$ is contained in
 $e L$, but since $\b v$ is primitive we must have $A\b v = e L$. Hence $A\b v$ is an orthogonal factor of $L$, which implies
 that $a$ divides $D$. Since $a$ and $D$ are relatively prime, this is possible only if $a$ is a unit.

\end{proof}

\begin{corollary}\label{C:rho} The numbers $\rho_i$ of Corollary \ref{C:gauss1} are given by
\[
\rho_i=\begin{cases} \displaystyle \frac{|R_i|}{|SO(L_i)|}\quad \textrm{if}\quad \deg(a)>0\\
\\
\displaystyle \frac{2 |R_i|}{|SO(L_i)|}\quad \textrm{if}\quad \deg(a)=0.\\
\end{cases}
\]
\end{corollary} 
\begin{proof}
The result follows immediately from the Lemma \ref{L:freeaction} and the orbit-stabilizer theorem.
\end{proof}
We shall now describe the right-hand side of \eqref{E:gauss1}. Recall that $B\simeq A[\lambda]$, where $\lambda^2=-aD$.
To simplify the notation, let $C=B[1/D]$. There is a natural 
map $\Pic(B)\to \Pic(C)$, which is readily seen to be surjective. We shall investigate its kernel.

\begin{lemma}\label{L:units} With the notation above:

\begin{enumerate} 
\item If $\deg(a)>0$ then $C^\times=A[1/D]^\times.$ 
\item If $\deg(a)=0$ then $C^\times=A[1/D]^\times \langle \lambda \rangle$, where $\langle \lambda \rangle$
is the multiplicative subgroup generated by $\lambda$.
\end{enumerate}
\end{lemma}

\begin{proof} Let $z\in C^\times$. Multiplying $z$ by a suitable unit of $A[1/D]$, we can assume
without loss of generality that $z$ is of the form $z=x+y\lambda$ with $x,y\in A$ and 
$\gcd(x,y,D)=1$. Since $z$ is a unit, its norm $e:=N(z)=x^2+a D y^2$ is a product of primes 
dividing $D$. If $p$ is a prime divisor of $e$,  then $p$ divides $x$ and we get the equation 
$ p x_0^2+aD_0 y^2=e_0$, where $x_0=x/p$, $D_0=D/p$ and $e_0=e/p$. If $p\mid e_0$, then $p\mid D_0$
since $a$ and $D$ are relatively prime, but this is a contradiction with the fact that $D$ is square-free.
Hence $e$ is square-free and therefore $e$ divides $D$. Letting $x_1=x/e$ and $D_1=D/e$ in the
original equation, we get $e x_1^2 +a D_1 y^2=1$. Since the form $e x_1^2 +a D_1 y^2$ is definite 
($-aD$ is not a square in $K_\infty$) we have that either $e$ and $x_1$ are units and $y=0$ or 
$y$ and $a D_1$ are units and $x_1=0$. In the first case $z=x\in A^\times$. The second case 
holds only if $a$ is a unit  and in this situation
$z=y\lambda$ with $y\in A^\times$.

\end{proof}

\begin{lemma} Let $r$ be the number of prime divisors of $D$. 
The kernel of the natural map $\Pic(B)\to \Pic(C)$ is an elementary $2$-group of
order $2^r$ if $\deg(a)>0$ and of order $2^{r-1}$ if $\deg(a)=0$. In particular
\begin{equation}\label{E:order_pic} 
|\Pic(C)|=\begin{cases} 2^{-r} |\Pic(B)|\ \mathrm{if}\ \deg(a)>0;\\
2^{-r+1}|\Pic(B)|\ \mathrm{if}\ \deg(a)=0.
\end{cases}
\end{equation}

\end{lemma}
\begin{proof} We shall reduce the proof to the case where $a$ is square-free, or, equivalently, 
to the case where the orders involved are integrally closed. 

Let $\bbar{B}$ and $\bbar{C}$ be the integral closures of $B$ and $C$ in $E$ respectively.
Write $a=a_0 a_1^2$, where $a_0$ is square-free. We have a commutative diagram
where the rows are exact  and the vertical maps $f$ and $\bbar{f}$ are onto.

\[
\minCDarrowwidth20pt
\begin{CD}
1@>>>\bbar{B}^\times/B^\times @>>>\displaystyle \prod_{p\mid a_1} \bbar{B}_\pp^\times/{B}_\pp^\times
@>>>\Pic(B) @>>>\Pic(\bbar{B})@>>>1\\
@. @VVV @| @VVfV @VV\bbar{f}V\\
1@>>>\bbar{C}^\times/C^\times @>>>\displaystyle\prod_{p\mid a_1} \bbar{C}_\pp^\times/{C}_\pp^\times
@>>>\Pic(C) @>>>\Pic(\bbar{C})@>>>1.
\end{CD}
\]
Note that since $a$ and $D$ are relatively prime, $B_\pp=C_\pp$ for $p\mid a_1$.  Clearly $\bbar{B}^\times={B}^\times=A^\times$ since $\deg(D)>0$.
Since $\bbar{C}=A[1/D, \lambda_0]$, where $\lambda_0^2=-a_0D$, by Lemma \ref{L:units} we have $\bbar{C}^\times/C^\times=\{1\}$ except when $\deg(a_0)=0$ and $\deg(a)>0$, in which case we have $\bbar{C}^\times/C^\times = \langle \lambda_0 \rangle/\langle \lambda_0^2 \rangle \simeq \Z/2\Z$. It follows immediately from the diagram that
\begin{equation}\label{E:kerf}
|\ker f|=\begin{cases} 2|\ker \bbar{f}| & \ \mathrm{if}\  \deg(a_0)=0\ \mathrm{and}\ \deg(a)>0;\\
|\ker \bbar{f}|& \ \mathrm{in\ all\ other\ cases.}
\end{cases}
\end{equation}
We shall now determine $\ker \bbar{f}$. Let $p_i$ ($i=1,\ldots,r$) be the prime divisors of $D$. Since
$\bbar{B}=A[\lambda_0]$, where $\lambda_0^2=a_0 D$, these primes are ramified in $\bbar{B}$, i.e. 
$p_i \bbar{B}=\gotik{p}_i^2$, where $\gotik{p}_i$ is a prime of $\bbar{B}$. The elements of 
$\ker \bbar{f}$ are represented by ideals $\mathfrak{a}$ of $\bbar{B}$ satisfying $\mathfrak{a} \bbar{C}=\bbar{C}$,
so such an ideal is a product of primes diving $D$. This shows that 
$\ker \bbar{f}$ is a $2$-elementary abelian group generated by the classes of the $\gotik{p}_i$ in $\Pic(\bbar{B})$,
which we will denote by $[\gotik{p}_i]$.

Any relation among the $[\gotik{p}_i]$ must be of the form
 \begin{equation*}
  [\mspace{1mu} \gotik{p}_{i_1}] [\mspace{1mu} \gotik{p}_{i_2}] \dots [\mspace{1mu}\gotik{p}_{i_k}] = 1,
 \end{equation*}
 where $\{i_1, i_2, \dots, i_k\} \subset \{1, 2, \dots, r\}$. Equivalently,
 \begin{equation}\label{E:relation}
  \gotik{p}_{i_1} \gotik{p}_{i_2} \dots \gotik{p}_{i_k} = \omega \bbar{B}
 \end{equation}
 for some $\omega \in \bbar{B}$. Since $\gotik{p}_{i}^2 = p_{i} \bbar{B}$,
 it follows that $\omega^2 u \in A$ for some $u \in \bbar{B}^{\times}$. But $\bbar{B}^\times = A^{\times}$
 since $-a_0 D$ is not a square in $K_{\infty}$ and $\deg(D)>0$. Consequently, $\omega^2 \in A$ and $\omega^2 | a_0 D$.
 The second condition ensures that $\omega\not\in A$, so $\omega=b \lambda_0$ for some $b\in A$.
 Then we also have $ a_0 D|\omega^2$; this is possible only if $\{i_1, i_2, \dots, i_k\} = \{1, 2, \dots, r\}$
 and $\deg a_0=0$. Thus
\begin{equation}\label{E:kerbarf}
|\ker \bbar f|=\begin{cases} 2^{r-1} & \ \mathrm{if}\  \deg(a_0)=0;\\
2^r & \ \mathrm{if}\  \deg(a_0)>0.
\end{cases}
\end{equation}

We finish by putting together \eqref{E:kerf} and \eqref{E:kerbarf}.

\end{proof} 

We can now state and prove the formula for the weighted sum of primitive representation numbers. For a definite $A$-lattice $(L,Q)$ and
a polynomial $a\in A$, we denote by $R(L,a)$ the number of solutions of the equation
$Q(\b x)=a$ with $\b x$ {\em primitive} in $L$.

\begin{theorem}\label{T:siegel} Let $D$ be a square-free non-constant polynomial and let $(V,Q)$ be a ternary quadratic space
over $K$ of determinant $D$. Let
$L_1,\ldots, L_h$ be representatives of all the isometry classes of integral $A$-lattices in $V$
of determinant $D$ and let $a$ be a polynomial relatively prime to $D$ such that $- aD\not\in {K_\infty^\times}^2$. Then
\begin{equation}\label{E:gauss2}
\sum_{i=1}^h \frac{R(L_i,a)}{| SO(L_i)|}=2^{-r}|\Pic(A[\sqrt{-a D}])|,
\end{equation}
where $r$ is the number of prime factors of $D$.
\end{theorem}

\begin{proof} The formula follows from Corollaries \ref{C:rho} and \ref{C:gauss1} together with \eqref{E:order_pic}.
The case $\deg(a)=0$ requires special consideration in the computations, but the end
formula turns out to be the same in both cases.
\end{proof}

\section{The Epstein Zeta Function}

In this and the next sections we will encounter functions of a complex variable $s$ 
that are naturally functions of $q^{-s}$. If $F$ is such a function, we shall
write $F^*$ for the unique function such that $F(s)=F^*(q^{-s})$.

As in previous sections, $(V,Q)$ denotes a definite ternary quadratic space over $K$. Let $L\subset V$ be 
an integral $A$-lattice. The {\em Epstein zeta function} of $L$ is
defined by
\begin{equation}\label{E:Epsteinzeta}
Z_L(s)=\sum_{{\b x}\in L\setminus\{0\}} |Q({\b x})|^{-s}.
\end{equation}
Collecting the terms of the same degree, we have
\[
Z_{L}(s)=\sum_{k=0}^\infty \alpha_k(L) q^{-ks},
\]
where 
\[
\alpha_k(L)=|\{{\b x}\in L\: : \:\deg Q({\b x})=k\}|.
\]
\begin{proposition}\label{P:zcoeff}
Let $\delta= \deg\det(L)$. For $k> \delta$ we have
\[
\alpha_k(L)=\begin{cases} (1-q^{-1})q^{(3k-\delta+4)/2}  & \mathrm{if}\quad k\equiv \delta\pmod 2\\
(1-q^{-2}) q^{(3k-\delta+5)/2}  & \mathrm{if}\quad k\not\equiv \delta\pmod 2.
\end{cases}
\]
\end{proposition}
\begin{proof} Let $\mu_i$ ($i=1,2,3$) be the successive minima of $L$ and let $L_k=\{{\b x}\in L: \deg Q({\b x})\le k\}$.
It is easy to see that for $k\ge \delta$
\[
\dim _{\F_q} L_k=\lfloor (k-\mu_1)/2\rfloor +\lfloor (k-\mu_1)/2\rfloor +\lfloor (k-\mu_1)/2\rfloor +3
\]  (see \cite[\S 3]{Bureau:2009kx}). Since $\delta=\mu_1+\mu_2+\mu_3$, we have
\[
|L_k|=\begin{cases}
q^{(3k-\delta+4)/2}  & \mathrm{if}\quad k\equiv \delta\pmod 2\\
q^{(3k-\delta+5)/2}  & \mathrm{if}\quad k\not\equiv \delta\pmod 2.
\end{cases}
\]

The proposition follows from the simple observation that $\alpha_k(L)=|L_k|-|L_{k-1}|$.

\end{proof}
\begin{corollary} $Z_L(s)$ is a rational function of $u=q^{-s}$ with simple poles at $u=\pm q^{-3/2}$.
\end{corollary}
\begin{proof}
Let $Z_L^*(u)=Z_L(s)$ and let $P(u)=\sum_{k=0}^{\delta-1} \alpha_k(L) u^k$. Thus we have, using Proposition \ref{P:zcoeff},
\[
\begin{aligned} 
Z_L^*(u)& =P(u) + q^{\delta+2} (1-q^{-1})u^\delta \sum_{l=0}^\infty (q^3 u^2)^l +  q^{\delta+4}(1-q^{-2}) u^{\delta+1}\sum_{l=0}^\infty (q^3 u^2)^l\\
& =P(u) + q^\delta (q-1) q \frac{(1 + q u + q^2 u) u^\delta}{1-q^3 u^2}.
\end{aligned}
\]
\end{proof}

\begin{definition} We shall say that two formal power series $f(u)$ and $g(u)$ in $u$ are {\em equivalent} if
$f(u)-g(u)$ is a polynomial in $u$. We shall denote this equivalence relation by $f(u)\sim g(u)$. Similarly,
two meromorphic functions
$F(s)$ and $G(s)$ will be said to be {\em equivalent} if $F(s)-G(s)$ is a polynomial in $q^{-s}$.
\end{definition}

\begin{proposition}\label{P:index}
Let $M$ and $L$ be $A$-lattices in $V$ and assume $L\subset M$. Then
\[
Z_M(s) \sim [M:L]\ Z_L(s),
\]
where $[M:L]$ is the index of $L$ in $M$.
\end{proposition}
\begin{proof} Let $k$ be large enough so that the restriction of the canonical projection $M\to M/L$ to
$M_k$ is onto. Then $[M_k: L_k]=[M:L]$ for all $k$ large enough. It follows that 
$\alpha_k(M)=|M_k|-|M_{k-1}|=[M:L] (|L_k|-|L_{k-1}|)=[M:L]\alpha_k(L)$ for $k$ large enough. 
\end{proof}

\begin{definition} Let $\chi: L\to \C$ be any function. The {\em Epstein zeta function twisted by $\chi$} is
defined by
\begin{equation}
Z_L(s,\chi) = \sum_{{\b x}\in L\setminus\{0\}} \chi({\b x}) |Q({\b x})|^{-s}
\end{equation}
\end{definition}

\begin{lemma}\label{L:epstein-d}
Let $L\subset V$ be an integral lattice of square-free determinant $D$. For $d|D$, let $\chi_d$
be the characteristic function of the set $\{{\b x}\in L \: : \:Q({\b x})\equiv 0\pmod d\}$ and let $d_0$ be the
product of the primes dividing $d$ at which $Q$ is isotropic and let $d_1$ be the product of the primes dividing $d$ at which $Q$ is anisotropic.
Then
\[
Z_L(s,\chi_d)\sim (2 |d_0|^{-1}|d_1|^{-2}-|d|^{-2}) Z_L(s).
\]
\end{lemma}
\begin{proof} Let $\pp$ be a prime divisor of $D$. If $Q$ is isotropic
at $\pp$ then $L_\pp$ has a basis $\{{\b v}_1,{\b v}_2,{\b v}_3\}$ over $A_\pp$ such that $Q(x{\b v}_1+y{\b v}_2+z {\b v}_3)=xy-Dz^2$. 
For each $d|D$ we define two sublattices $L^{(d,i)}$ ($i=1,2$) of $L$ by their local components as follows
\[
L^{(d,i)}_\pp=\begin{cases} A_\pp {\b v}_i +\pp L_\pp^\sharp &\ \mathrm{if}\  \pp\mid d_0;\\
\pp L_\pp^\sharp &\ \mathrm{if}\  \pp\mid d_1;\\
L_\pp &\ \mathrm{if}\  \pp\nmid d.
\end{cases}
\]
It is easy to see from the definition of $L^{(d,i)}$ that for $\b x\in L$
\[
Q({\b x})\equiv 0\pmod d \iff {\b x}\in L^{(d,1)}\cup L^{(d,2)}.
\]
Hence
\[
Z_L(s,\chi_d)(s)=Z_{L^{(d,1)}}(s)+Z_{L^{(d,2)}}(s)-Z_{L^{(d,1)}\cap  L^{(d,2)}}(s).
\]

Since $d=d_0 d_1$ we also have
\[
[L: L^{(d,i)}]=|d_0||d_1|^2\quad\mathrm{and}\quad [L: L^{(d,1)}\cap  L^{(d,2)}]=|d|^2.
\]
We finish by applying Proposition \ref{P:index}.
\end{proof}

For $d\in A\setminus\{0\}$ we define 
\begin{equation}\label{E:functionM}
M_d(s)=\sum_{e|d} \mu(e) |e|^{-s},
\end{equation}
where $\mu$ is the M{\"o}bius function. It is easy to see that $M_d(s)$ has a product decomposition
\[
M_d(s)=\prod_{p|d} (1-|p|^{-s}).
\]
The function $M_d(s)$ defined in \eqref{E:functionM} will play an important role in many subsequent computations.

\begin{lemma}\label{L:epstein_phi1} Write $D=D_0 D_1$, where $D_0$ is the product of the isotropic primes dividing $D$ and
$D_1$ is the product of the anisotropic primes dividing $D$. Let $\psi$ be the characteristic function of the set $\{{\b x}\in L\: : \:\gcd(Q({\b x}),D)=1\}$. 
Then
\[
Z_L(s,\psi)\sim \left(2M_{D_0}(1)M_{D_1}(2)-M_D(2)\right) Z_L(s).
\]

\end{lemma}
\begin{proof} Recall that $\chi_d$ denotes the characteristic function of the set \[
\{{\b x}\in L\setminus \{0\}\: : \: Q({\b x})\equiv 0 \pmod{d}\}.
\]
Notice that $\chi_d$ as a function of $d$ is multiplicative, i.e. if $(d,e)=1$ then $\chi_{de}=\chi_{d}\chi_e$.
The functions $\psi$ and $\chi_d$ are related by the identity
\[
\begin{aligned}
\psi=&\prod_{\scriptsize\begin{matrix} p|D\\ p\ \mathrm{prime} \end{matrix}} (1-\chi_p)\\
=&\sum_{d|D} \mu(d) \chi_d.
\end{aligned}
\]
It follows immediately that
\[
Z_L(s,\psi)=\sum_{d|D}\mu(d)Z_L(s,\chi_d).
\]
Hence, by Lemma \ref{L:epstein-d}, we have
\[
\begin{aligned}
Z_L(s,\psi) &\sim \sum_{\scriptsize\begin{matrix} d_0|D_0\\ d_1|D_1\end{matrix}} \mu(d_0 d_1) (2 |d_0|^{-1}|d_1|^{-2}-|d|^{-2}) Z_L(s)\\
&=\left(2M_{D_0}(1)M_{D_1}(2)-M_D(2)\right) Z_L(s).
\end{aligned}
\]

\end{proof}

Recall that the {\em content} of a vector ${\b x}\in L\setminus \{0\}$ is the monic polynomial $c({\b x})\in A$
such that $c({\b x}) A {\b x} =K{\b x}\cap L$. A vector ${\b x}\in L$ is {\em primitive} if $c({\b x})=1$. 

\begin{lemma}\label{L:epstein_phi2}
Let $\phi$ be the characteristic function of the set $\{{\b x}\in L \: : \:{\b x}\ \mathrm{primitive} \}$. Then
\[
Z_L(s,\psi)= M_D(2s) \zeta(2s) Z_L(s,\phi\psi),
\]
where $\zeta$ is the zeta function of $A$.
\end{lemma}
\begin{proof}

Collecting the terms of equal content in the defining sum for $Z_L(s,\psi)$ we have
\[
\begin{aligned} 
Z_L(s,\psi)&= \sum_{(d,D)=1} \sum_{c({\b x})=d} \psi({\b x})|Q({\b x})|^{-s}\\
&=\sum_{(d,D)=1} \sum_{c(\b y)=1} \psi(\b y)|Q(\b y)|^{-s}|d|^{-2s}\\
&=(\sum_{(d,D)=1} |d|^{-2s}) Z_L(s,\phi\psi)\\
&= \left(\prod_{p|D} (1-|p|^{-2s})\right ) \zeta(2s) Z_L(s,\phi\psi)\\
&=M_D(2s) \zeta(2s) Z_L(s,\phi\psi).
\end{aligned}
\]

\end{proof} 
In order to prove the main theorem of this section, we need the following easy lemma from complex analysis.

\begin{lemma}\label{L:complex}
Let $g$ be a function analytic on the disk $|z|<R$, where $R>1$, and let $g^+(z)=(g(z)+g(-z))/2$
and $g^-(z)=(g(z)-g(-z))/2$. Let $f(z)=g(z)/(1-z^2)$ and consider the Taylor series expansion of $f(z)$ at 
$z=0$:
\[
f(z)=\sum_{n=0}^\infty a_n z^n \quad\mathrm{for}\ |z|<1.
\]
Then the sequences $\{a_{2k}\}$ and $\{a_{2k+1}\}$ converge and
\[
\lim_{k\to \infty} a_{2k}= g^+(1)\quad \mathrm{and}\quad  \lim_{k\to \infty} a_{2k+1}=g^-(1).
\]
\end{lemma}
\begin{proof} We leave the proof as an exercise.
\end{proof}

We are now ready to prove the main result of this section.

\begin{theorem}\label{T:coeff-beta} Let $\beta_k(L)=\# \{{\b x}\in L\setminus\{0\}\: : \:\deg Q({\b x})=k,\ (Q({\b x}),D)=1,\ c({\b x})=1\}$. Then the sequences
$\beta_{\delta +2 m}(L) q^{-3m}$ and $\beta_{\delta +2 m+1}(L)q^{-3m}$ converge and 
\[
\begin{aligned}
\lim_{m\to \infty} & \frac{\beta_{\delta +2 m}(L)}{ q^{3m}} &=\frac{2M_{D_0}(1)M_{D_1}(2)-M_D(2)}{M_D(3)\zeta(3)} (1-q^{-1})q^{\delta+2}\\
\lim_{m\to \infty} & \frac{\beta_{\delta +2 m+1}(L)}{ q^{3m}} &=\frac{2M_{D_0}(1)M_{D_1}(2)-M_D(2)}{M_D(3)\zeta(3)} (1-q^{-2})q^{\delta+4}.
\end{aligned}\]
\end{theorem}
\begin{proof} Notice that $\beta_k(L)$ is the coefficient of order $k$ of $Z_L^*(u,\phi\psi)$. 
To simplify notation, we let $C=2M_{D_0}(1)M_{D_1}(2)-M_D(2)$. Putting together Lemmas \ref{L:epstein_phi1}
and \ref{L:epstein_phi2}
we get
\[
M_D^*(u^2) \zeta^*(u^2) Z_L^*(u,\phi\psi)  = C Z_L^*(u)+ F(u)
\]
where $F(u)$ is a polynomial in $u$.  Notice that the zeros of $M_D^*(u^2)$ lie on the circle $|u|=1$ and that
$\zeta^*(u^2)=1/(1-qu^2)$ has no zeros. 
we get
\begin{equation}\label{E:ident}
Z_L^*(u,\phi\psi)=\frac{C Z_L^*(u)}{M_D^*(u^2) \zeta^*(u^2)} + G(u),
\end{equation}
where $G(u)$ is analytic on the disk $|u|<1$. The only poles of
$Z_L^*(u,\phi\psi)$ on this disk are $u=\pm q^{-3/2}$.

Notice
that $G(q^{-3/2} z)$ is analytic on the disk $|z|<q^{3/2}$, so its Taylor series at the origin converges
for $z=1$. Hence the Taylor coefficients of $G(q^{-3/2} z)$ tend to $0$.  Multiplying 
\eqref{E:ident} throughout by $u^\delta$ and applying Lemma \ref{L:complex} to $u^\delta Z_L^*(u)/(M_D^*(u^2) \zeta^*(u^2))$ 
with $u=q^{-3/2} z$, we get

\[
\begin{aligned}
\lim_{m\to \infty} \frac{\beta_{\delta+2m}(L)}{q^{3m}}& = \frac{C}{M_D(3)\zeta(3)}\lim_{m\to \infty}\frac{\alpha_{\delta+2m}(L)}{q^{3m}}
\\
& = \frac{C}{M_D(3)\zeta(3)} (1-q^{-1})q^{\delta+2}.
\end{aligned}
\]

Similarly,

\[
\begin{aligned}
\lim_{m\to \infty} \frac{\beta_{\delta+2m+1}(L)}{q^{3m}}& = \frac{C}{M_D(3)\zeta(3)}\lim_{m\to \infty}\frac{\alpha_{\delta+2m+1}(L)}{q^{3m}}
\\
& = \frac{C}{M_D(3)\zeta(3)} (1-q^{-2})q^{\delta+4}.
\end{aligned}
\]

\end{proof}

\section{Averages of class numbers}

Let $D\in A$ be a square-free polynomial of degree $\delta$. Throughout this section, the symbol $\sum^\bullet$ will denote a sum 
restricted to polynomials relatively prime to $D$.

Define
\[
\Psi_D (k,l)=|\{(x,y)\in A^2\, :\, x,y\ \mathrm{monic},\ \gcd(x,y)=1\ \mathrm{and}\ \gcd(D,xy)=1\}|.
\]
\begin{lemma}\label{L:Psi_series} The following identity holds in the ring of iterated power series\\
$\Q[[u]][[v]]$
\begin{equation}
\sum_{k,l\ge 0} \Psi_D (k,l) u^k v^l=\frac{M^*_D(u) M^*_D(v)(1-q uv)}{M^*_D(uv)(1-q u)(1-q v)}.
\end{equation}
\end{lemma}

\begin{proof} Define
\[
\eta_D(s)=\sum_{a\ \textrm{monic}}^\bullet |a|^{-s}.
\]
It follows immediately from the Euler product decomposition of $\zeta(s)$ that
we have a factorization
\begin{equation}\label{E:eta}
\eta_D(s)=M_D(s) \zeta(s),
\end{equation}
where $M_D(s)$ is the polynomial in $q^{-s}$ defined in \eqref{E:functionM}. 

With this notation, we have
\[
\begin{aligned}
\eta_D(s)\eta_D(t)&=\sum_{a, b\ \textrm{monic}}^\bullet |a|^{-s}|b|^{-t}\\
&=\sum_{d\ \textrm{monic}}^\bullet \sum_{\scriptsize \begin{matrix} a, b\ \textrm{monic}\\
\gcd(a,b)=d\end{matrix}}^\bullet |a|^{-s}|b|^{-t}\\
&=\sum_{d\ \textrm{monic}}^\bullet |d|^{-(s+t)}\sum_{\scriptsize \begin{matrix} a, b\ \textrm{monic}\\
\gcd(a,b)=1\end{matrix}}^\bullet |a|^{-s}|b|^{-t}\\
&=\eta_D(s+t) \sum_{\scriptsize \begin{matrix} a, b\ \textrm{monic}\\
\gcd(a,b)=1\end{matrix}}^\bullet |a|^{-s}|b|^{-t}.\\
\end{aligned}
\]
Letting $u=q^{-s}$ and $v=q^{-t}$ we get
\[
\eta_D^*(u)\eta_D^*(v)=\eta_D^*(uv)\sum_{k,l\ge 0} \Psi_D (k,l) u^k v^l,
\]
which, combined with \eqref{E:eta}, shows the Lemma.

\end{proof}

For the next computation, we shall need the following variant of
Lemma \ref{L:complex}.

\begin{lemma}\label{L:complex2}
Let $g$ be an analytic function on the disk $|z|<R$, where $R>1$. Let $f(z)=g(z)/(1-z)$ and consider the Taylor series expansion of $f(z)$ at 
$z=0$:
\[
f(z)=\sum_{n=0}^\infty a_n z^n \quad\mathrm{for}\ |z|<1.
\]
Then the sequence $\{a_n\}$ converges and $
\lim_{n\to \infty} a_{n}= g(1)$.
\end{lemma}
\begin{proof} We leave the proof as an exercise.
\end{proof}
\begin{lemma}\label{L:sum}
\begin{equation}\label{E:sum}
\lim_{l\to\infty} q^{-l}  \sum_{k=0}^\infty \Psi_D (k,l) q^{-2 k}= \frac{M_D(1)M_D(2)}{M_D(3)}\cdot\frac{\zeta(2)}{\zeta(3)}.
\end{equation}

\end{lemma}

\begin{proof} Let $u=q^{-2}$ and $v=q^{-1} z$, where $z$ is a complex variable, and substitute in \eqref{L:Psi_series}. 
We get 
\begin{equation}\label{E:Psi_series2}
\sum_{l=0}^\infty \left [q^{-l}\sum_{k=0}^{\infty} \Psi_D (k,l) q^{-2k}\right ] z^l= \frac{M^*_D(q^{-2}) M^*_D(q^{-1}z)(1-q^{-2} z)}{M^*_D(q^{-3} z)(1-q^{-1})(1-z)}.
\end{equation}
(It is easy to see that the inner sum converges uniformly in $l$ since $\Psi_D (k,l) q^{-l}$ is bounded by $q^{k}$.)
Notice that the function on the right-hand side of \eqref{E:Psi_series2} has a simple pole at $z=1$ and 
that all other poles lie on the circle $|z|=q^3$. Hence, applying Lemma \ref{L:complex2}, we get
\[
\begin{aligned}
\lim_{l\to\infty} q^{-l}  \sum_{k=0}^\infty \Psi_D (k,l) q^{-2 k}=& \left. \frac{M^*_D(q^{-2}) M^*_D(q^{-1}z)(1-q^{-2} z)}{M^*_D(q^{-3} z)(1-q^{-1})}\right|_{z=1}\\
=& \frac{M_D(1)M_D(2)}{M_D(3)}\cdot\frac{\zeta(2)}{\zeta(3)}.
\end{aligned}
\]
\end{proof}

For each $b\in A\setminus\{0\}$, we define 

\[
\chi_b(a)=\legendre{b}{a}
\]
for all monic polynomials $a$, where $\legendre{b}{a}$ is the Legendre symbol. 
The Dirichlet $L$-series $L(s,\chi_b)$ attached to $\chi_b$ is defined by
\[
L(s,\chi_b)=\sum_{a\ \mathrm{monic}} \chi_b(a) |a|^{-s}.
\]
If $b$ is not a square, $L(s,\chi_b)$ is actually a polynomial of
degree $\le \deg b-1$ in $q^{-s}$ (see \cite[Proposition 4.3]{Rosen:2002kv}). In this case we write

\[
L(s,\chi_b)=\sum_{k=0}^{\deg b-1} c_k(\chi_b)q^{-s k},
\]
where
\[
c_k(\chi_b)=\sum_{\scriptsize \begin{matrix} a\ \mathrm{monic} \\ \deg a=k \end{matrix}} \chi_b(a).
\]

Our task will be to compute sums of coefficients $c_k(\chi_{D  m})$ as $m$ runs
over all polynomials of fixed degree $l$ prime to $D $. We have

\begin{equation}
\sum_{\deg m=l}^\bullet c_k(\chi_{D  m})=\sum_{\scriptsize \begin{matrix} a\ \mathrm{monic} \\ \deg a=k \end{matrix}}^\bullet \sum_{\deg m=l}^\bullet \legendre{m}{a}\legendre{D }{a}.
\end{equation}
(Recall that a bullet ${\ }^\bullet$ above the summation symbol indicates that the sum
is restricted to polynomials prime to $D $.)

Assume that $\deg m=l\ge k+\delta$. 
If $a$ is not a square, then 
\[
\sum_{\deg m=l}^\bullet \legendre{m}{a}=0,
\]
since in this case $\legendre{}{a}$ is a nontrivial
character on $(A/D  a A)^\times$. If $a$ is a square, then
the summation above is just the number of polynomials $m$ of degree $l$
prime to $D $ and $a$. Thus

\begin{equation}
\sum_{\deg m=l}^\bullet c_k(\chi_{D  m})= (q-1) \Psi_{D }(k/2,l)\quad\mathrm{for}\ l\ge k+\delta,
\end{equation}
with the convention that $\Psi_{D }(k/2,l)=0$ if $k/2$ is not an integer.
The factor $q-1$ on the right-hand side accounts for the fact that
the sum includes all polynomials $m$ of degree $l$, not only the monic ones.

Thus we have so far

\begin{equation}\label{E:sumL}
\begin{split}
\sum_{\deg m=l}^\bullet L(s,\chi_{D  m})=(q-1)\sum_{k=0}^\floor{(l-\delta)/2}\Psi_{D }(k,l)q^{-2s k} 
\\ 
+\sum_{k=l-\delta+1}^{l+\delta-1} \left(\sum_{\deg m=l}^\bullet c_k(\chi_{D  m})\right)q^{-s k}.
\end{split}
\end{equation}

\begin{lemma}\label{L:binom} $|c_k(\chi_{D  m})|\le \binom{l+\delta-1}{k}q^{k/2}$.
\end{lemma}

\begin{proof} Let $L^*(u,\chi_{D  m})=L(s,\chi_{D  m})$, where $u=q^{-s}$, and write
$m=m_0 m_1^2$, where $m_0$ is square-free. Let $E=K(\sqrt{D  m})$. We get
from the Euler product factorization of  $L(s,\chi_{D  m})$  the identity

\[
L^*(u,\chi_{D  m})=f(u) L_E(u) \prod_{p\mid m_1}\left(1-\legendre{m_0D }{p}u^{\deg p}\right)
\]
where $f(u)=(1-u^2)$, $(1-u)$ or $(1-u)^2$ according to whether the prime at infinity of $K$
is inert, ramified or split in $E$ and $L_E(u)$ is the $L$-function of $E$.
The polynomial $L_E(u)$ has a factorization of the form
\[
L_E(u)=\prod_{i=1}^{2g}(1-\pi_i u)
\]
where the $g$ is the genus of $E$.  By the Riemann Hypothesis (see \cite[Theorem 5.10]{Rosen:2002kv}), the inverse roots $\pi_i$ of $L_E(u)$ satisfy $|\pi_i|=q^{1/2}$. 
Let $\pi_{2g+1},\ldots, \pi_{n}$ (where $n\le l+\delta-1$) be the inverses of the remaining
roots of  $L^*(u,\chi_{D  m})$. Notice that $|\pi_i|=1$ for $i=2g+1,\dots,n$. With this notation we have
\[
L^*(u,\chi_{D  m})=\prod_{i=1}^{n}(1-\pi_i u),
\]
so the coefficient $c_k(\chi_{D  m})$ is given by
\[
c_k(\chi_{D  m})=s_k(\pi_1,\cdots,\pi_n)
\]
where $s_k$ is the degree $k$ elementary symmetric function in $n$ variables.
Hence
\[
\begin{aligned}
|c_k(\chi_{D  m})|& \le \binom{n}{k} \sup\{|\pi_{i_{1}}\cdots \pi_{i_{k}}|\}\\
&\le \binom{n}{k} q^{k/2}\\
&\le \binom{l+\delta-1}{k}q^{k/2}.
\end{aligned}
\]
\end{proof}
We can now prove the main theorem of this section
\begin{theorem}\label{T:L_avg}
\begin{equation}\label{E:L_avg}
\lim_{l\to\infty} \frac{1}{q^l} \sum_{\deg m=l}^\bullet L(1,\chi_{D  m})=
 \frac{M_D(1)M_D(2)}{M_D(3)}q (1-q^{-2}).
\end{equation}
\begin{proof} 
We shall first estimate the second sum on the right-hand side of \eqref{E:sumL} at $s=1$ using Lemma \ref{L:binom}. Take
$l$ large enough so that $l-\delta+1>(l+\delta-1)/2$ ($l\ge 3\delta$ suffices). 
For ${l+\delta-1}\ge k\ge {l-\delta+1}$ we have
\[
\binom{l+\delta-1}{k}\le \binom{l+\delta-1}{l-\delta+1}=\binom{l+\delta-1}{2\delta-2}.
\]
Using Lemma \ref{L:binom} we get
\begin{equation}
\left|\sum_{k=l-\delta+1}^{l+\delta-1} \left(\sum_{\deg m=l}^\bullet c_k(\chi_{D  m})\right)q^{-k}\right|\le 
q^l(q-1) (2\delta-1)\binom{l+\delta-1}{2\delta-2} q^{-(l-\delta+1)/2}.
\end{equation}

Hence

\begin{equation}\label{E:tail_estimate}
\frac{1}{q^l}\left|\sum_{k=l-\delta+1}^{l+\delta-1} \left(\sum_{\deg m=l}^\bullet c_k(\chi_{D  m})\right)q^{-k}\right|\le 
P(l) q^{-l/2}
\end{equation}
where we have set
\[
P(l)=(q-1)q^{(\delta-1)/2} (2\delta-1)\binom{l+\delta-1}{2\delta-2}.
\]
Notice that $P(l)$ is a polynomial in $l$ of degree $2\delta-2$, so taking limits in \eqref{E:tail_estimate},
we get
\[
\lim_{l\to \infty} \frac{1}{q^l}\sum_{k=l-\delta+1}^{l+\delta-1} \left(\sum_{\deg m=l}^\bullet c_k(\chi_{D  m})q^{-k}\right)=0.
\]
Therefore, from \eqref{E:sumL}
\begin{equation}\label{E:semifinal}
\lim_{l\to\infty} \frac{1}{q^l}\sum_{\deg m=l}^\bullet L(1,\chi_{D  m})=(q-1)\lim_{l\to\infty} \frac{1}{q^l}\sum_{k=0}^\floor{(l-\delta)/2}\Psi_{D }(k,l)q^{-2 k}.
\end{equation}

In order to conclude, we inspect the ``tail'' of the series in \eqref{E:sum}. Using the obvious inequality
$\Psi_D (k,l)\le q^{k+l}$ we get
\[
\frac{1}{q^l}\sum_{k=\floor{(l-\delta)/2}+1}^\infty \Psi_D (k,l) q^{-2 k}\le \sum_{k=\floor{(l-\delta)/2}+1}^\infty q^{-k}.
\]
The right-hand side of the above inequality  tends to zero as $l\to\infty$. The theorem follows from \eqref{E:semifinal} and
\eqref{E:sum}.

\[
\begin{aligned}
\lim_{l\to\infty} \frac{1}{q^l}\sum_{\deg m=l}^\bullet L(1,\chi_{D  m})&=(q-1)\frac{M_D(1)M_D(2)}{M_D(3)}\cdot\frac{\zeta(2)}{\zeta(3)}\\
&= \frac{M_D(1)M_D(2)}{M_D(3)}q (1-q^{-2}).
\end{aligned}
\]

\end{proof}
\end{theorem}
\begin{remark}\label{R:L_avg} Theorem \ref{T:L_avg} can be restated in perhaps a more suggestive way
as a limit of averages. For $l\ge \delta$ the number of polynomials of degree $l$
prime to $D $ is $(q-1)q^l M_D(1)$. Thus \eqref{E:L_avg} becomes
\begin{equation}\label{E:L_avg2}
\lim_{l\to\infty} \frac{1}{(q-1)q^l M_D(1)} \sum_{\deg m=l}^\bullet L(1,\chi_{D  m}) =
\frac{M_D(2) \zeta(2) }{M_D(3) \zeta(3)}.\\
\end{equation}
(Compare with \cite[Theorem 1.4]{Hoffstein:1992ao}.)
\end{remark}

\begin{corollary}\label{C:classno_avg} Let $h(mD )=|\Pic(A[\sqrt{mD }]|$. Then
\begin{equation}
\lim_{l\to \infty}\frac{1}{q^{3l}} \sum_{\deg m=\delta+2l+1}^\bullet h(mD ) = q^{2\delta} (q^2-1)\frac{M_D(1)M_D(2)}{M_D(3)}.
\end{equation}
\end{corollary}
\begin{proof} For $\deg m= \delta +2l+1$  we have the relation (see e.g. \cite[Theorem 0.6 (i)]{Hoffstein:1992ao})
\[
\begin{aligned}
L(1,\chi_{m D })=& \frac{q^{1/2}}{|mD|^{1/2}} h(mD )\\
&=q^{-l-\delta}  h(mD ).
\end{aligned}
\]
Thus
\[
\begin{aligned}
\lim_{l\to \infty}\frac{1}{q^{3l}} \sum_{\deg m=\delta+2l+1}^\bullet h(mD )&= q^{2\delta+1} \lim_{l\to\infty} \frac{1}{q^{\delta+2l+1}}\sum_{\deg m=\delta+2l+1}^\bullet L(1,\chi_{D  m}) \\
&=q^{2\delta} (q^2-1)\frac{M_D(1)M_D(2)}{M_D(3)}.
\end{aligned}
\]
\end{proof}
\section{The mass formula}

We can use the results established in the previous sections to obtain Siegel's Mass Formula 
in the case where the determinant $D$ is square-free.

\begin{theorem}\label{T:mass}
Let $D$ be a square-free polynomial of degree $\delta$ and let
$L_1,\ldots,L_h$ be a complete list of representatives of the isometry
classes of ternary definite lattices of determinant $D$. Then

\begin{equation}\label{E:mass}  
\sum_{i=1}^h \frac{1}{|SO(L_i)|}= \frac{q^{\delta} M_D(1) M_{D_0}(2)}{2^r(q^2-1)(2M_{D_0}(1) -M_{D_0}(2))},
\end{equation}
where $r$ is the number of prime divisors of $D$.

\end{theorem}

\begin{proof} 
We take sums in \eqref{E:gauss2} over all polynomials $a$ relatively prime to $D$ 
with $\deg a=\delta+2 l+1$ for a fixed positive
integer $l$ (we restrict ourselves to degrees of this form to ensure that $-a D$ is
not a square in $K_\infty$ so the equivalent conditions of Lemma \ref{L:eqconditions} are satisfied).  We get
\[
\sum_{\deg a=\delta+2 l+1}^\bullet \sum_{i=1}^h \frac{R(L_i,a) }{|SO(L_i)|}=
\frac{1}{2^r} \sum_{\deg a=\delta+2 l+1}^\bullet h(-a D)
\]
Interchanging the sums and using the notation of Theorem \ref{T:coeff-beta} we have 
\[
\frac{1}{q^{3 l}} \sum_{i=1}^h \frac{\beta_{\delta+2 l+1}(L_i)}{|SO(L_i)|}= \frac{1}{2^r} \frac{1}{q^{3 l}}
\sum_{\deg a=\delta+2 l+1}^\bullet h(-a D).
\]
The formula follows from Theorem \ref{T:coeff-beta} and Corollary \ref{C:classno_avg} by taking the limit
as $l\to \infty$.
\end{proof}
\begin{remark} In the special case where $D$ is {\em irreducible}, formula \eqref{E:mass} becomes
\begin{equation}\label{E:mass2}
\sum_{i=1}^h \frac{1}{|SO(L_i)|}=\frac{q^\delta-1}{2(q^2-1)},
\end{equation}
which is equivalent to the one proved by Gekeler \cite[5.11] {Gekeler:1992fk} using
Drinfeld modules.
\end{remark}

\section{Exact class numbers}

Theorem \ref{T:mass} allows us also to give {\em exact} class
numbers for definite forms of {\em irreducible} determinant $D$. We shall
restrict ourselves for simplicity to the case where the degree of $D$ is odd (a similar argument can be made
in the case where $D$ has even degree). 

Assuming $D$ is irreducible of odd degree,  all the definite ternary forms of determinant
$D$ are in the genus of the diagonal form $Q_0=\langle1,-\epsilon, -\epsilon D\rangle$,
where $\epsilon$ is a nonsquare in $\F_q$. Since $D$ is irreducible, a quadratic
form $Q$ of determinant $D$ either represents a unit or is indecomposable. Let
$h_\mathrm{dec}$ be the number of classes of decomposable forms of determinant $D$
and let $h_\mathrm{ind}$ be the number of indecomposable ones (then
$h=h_\mathrm{dec}+h_\mathrm{ind}$). If $Q$ is decomposable 
then $Q$ is of one of the types:
\begin{equation}\label{E:types}
Q=\langle1\rangle\perp (-\epsilon Q')\quad \mathrm{or}\quad  Q=\langle\epsilon\rangle\perp (-Q'')
\end{equation}
where $Q'$ is a binary form in the principal genus of determinant $D$
and $Q''$ is the in the principal genus of determinant $\epsilon D$. The only
class common to these two types is represented by the diagonal form $Q_0$. 

Kneser \cite{Kneser:1982lb} showed that the theory of Gauss for binary
quadratic forms holds in great generality, and we can apply it in
particular for forms over $A$ (see also \cite{Hellegouarch:1989rf} for
the specific case of polynomial rings) in order to count the
possibilities for $Q'$ and $Q''$. The proper equivalence classes of
binary quadratic forms of determinant $D$ in the principal genus are in
one-to-one correspondence with elements of the ideal class group
$\Pic(A[\sqrt{-D}])$ (this group has odd order since $D$ is irreducible of odd degree).
The number of improper nontrivial equivalence classes in the principal genus is then
$(|\Pic(A[\sqrt{-D}])|-1)/2$. Similarly for the forms of determinant
$\epsilon D$. Thus:

\begin{equation}\label{E:dec}
\begin{aligned}
h_\mathrm{dec}&= 1+
\frac{|\Pic(A[\sqrt{-D}])|-1}{2}+\frac{|\Pic(A[\sqrt{-\epsilon
D}])|-1}{2}\\ & =\frac{L_{-D}(1)+L_{-D}(-1)}{2},
\end{aligned}
\end{equation}
where $L_{-D}$ is the numerator of the zeta function (see \cite[Theorem 5.9]{Rosen:2002kv}) of the
hyperelliptic curve $y^2=-D$. Here we use the well-known fact that
$L_{-\epsilon D}(u)=L_{-D}(-u)$. Note that in the particular case of an
elliptic curve ($\deg D=3$) we have $h=h_\mathrm{dec}=q+1$.

\begin{lemma}\label{L:orderSO} Denote by $SO(Q)$ the group of integral rotations of $Q$.
Then
\begin{enumerate}
\item If $Q=Q_0$ then $|SO(Q)|=2(q+1)$.
\item If $Q$ is as in \eqref{E:types} with $Q'$ and $Q''$ indecomposable, then $|SO(Q)|=2$.
\item If $Q$ is indecomposable, then $|SO(Q)|=1$.
\end{enumerate}
\end{lemma}

\begin{proof} We leave the proof as an easy exercise.
\end{proof}

\begin{theorem}\label{T:exactclassno} Let $D$ be a square-free irreducible polynomial of odd degree $\delta$
and let $h$ be the number of isometry classes of definite
ternary forms of determinant $D$ and let $h_\mathrm{ind}$ the number of those
that are indecomposable. Then

\begin{align}
h=&\frac{1}{2}\left[1+ \frac{q(q^{\delta-1}-1)}{q^2-1}+\frac{L_{-D}(1)+L_{-D}(-1)}{2}\right];
\label{E:exactclassno}\\ h_\mathrm{ind}
= &\frac{1}{2}\left[1+ \frac{q(q^{\delta-1}-1)}{q^2-1}-\frac{L_{-D}(1)+L_{-D}(-1)}{2}\right].
\end{align}

\end{theorem}
\begin{proof} This is a straightforward application of Theorem \ref{T:mass} together with 
\eqref{E:dec} and Lemma \ref{L:orderSO}
\end{proof}

\begin{remark} In the case there $D$ is irreducible, the form $Q$ is anisotropic exactly at $D$ 
and at $\infty$, so, by Corollary \ref{C:maxorders}, the maximal lattices in $V$ are in one-to-one-correspondence
with the maximal orders of the quaternion algebra $C_0(V,Q)$. So the number $h$ of Theorem \ref{T:exactclassno} is also
the {\em type} of the quaternion algebra $C_0(V,Q)$, i.e. the number of conjugacy classes of maximal orders
of $C_0(V,Q)$. Formula \eqref{E:exactclassno} is therefore equivalent to the formula proved by Gekeler \cite[Theorem 7.6]{Gekeler:1992fk}
using Drinfeld modules  for the type of a quaternion algebra ramified at $D$ and $\infty$.
\end{remark}

\end{document}